\documentclass{amsart}
 \usepackage{amsmath}
 \usepackage{epsfig} 
 \usepackage{graphics} 
 \usepackage[normalem]{ulem}
 \usepackage{color}

\input xy 
\xyoption{all}

\def\C{\mathbb C}
\def\P{\mathbb P}

\def\bcases{\begin{cases}}

\def\ecases{\end{cases}}

\newcommand{\e}{\epsilon}

\newcommand{\N}{\mathbb N}

\newcommand{\set}[1]{\left\{ #1\right\}}
\def\sm{\setminus}

\newcommand{\W}{\Omega}

\newcommand{\Z}{\mathbb Z}

\newtheorem{thm}{Theorem}
\newtheorem{main theorem}{Main Theorem}

\newtheorem{lemma}{Lemma}
\newtheorem{prop}{Proposition}

\newtheorem{problem 1}{Problem 1}
\newtheorem{problem 2}{Problem 2}
\newtheorem{problem 3}{Problem 3}
\theoremstyle{definition}

\newtheorem{defn}{Definition}

\newcommand{\bea}{\begin{eqnarray*}}
\newcommand{\eea}{\end{eqnarray*}}

\newcommand{\be}{\begin{equation}}
\newcommand{\ee}{\end{equation}}
\begin{document}

\title[Tautness and Fatou components in $ \P^2 $]
{Tautness and Fatou components in $ \P^2 $}

\author{Han Peters and Crystal Zeager}

\thanks{The first author was supported by a SP3-People Marie Curie Actionsgrant in the project Complex Dynamics (FP7-PEOPLE-2009-RG, 248443). The second author the was supported by NSF RTG grant DMS-0602191.}
\today

\begin{abstract}
We investigate the tautness of invariant Fatou components for holomorphic endomorphisms of $\P^2$. Previously, only basins of attraction were known to be taut. We show that two other kinds of recurrent Fatou components are taut. In the first of these cases, as well as for basins of attraction, we show that the Fatou components are in fact Kobayashi complete, which implies tautness.
 \end{abstract}

\maketitle

\begin{section}{introduction}

Although the study of Fatou components for complex dynamical systems in several variables has received quite a bit of interest in recent years, many elementary questions have not yet been answered. Here we will investigate one of these questions, namely whether the Fatou components of a holomorphic endomorphism of $\P^2$ are taut. This question was raised by Abate after he classified the dynamical behavior of holomorphic self-maps of taut domains \cite{Abate}. So far the only Fatou components for which the question is settled are basins of attraction, \cite{Weickert}. In this article we will prove tautness for several other classes of Fatou components. Before we go into details we will recall a few relevant results and definitions.

\begin{defn} A family $\mathcal{F} \subset \mathcal{O}(X, \W)$ is called normal if every sequence $ (f_n) \subset \mathcal{O}(X,\W)$ either has a subsequence $ (f_{n_i}) $ converging uniformly on compacts to $ f \in  \mathcal{O}(X,\W) $, or has a subsequence $ (f_{n_i}) $ which is \emph{compactly divergent}. A subsequence $ (f_{n_i}) $ is called compactly divergent if for any two compact sets $ K \subset X$ and $ L \subset \W $ there exists $ I \in \N$ so that for all $ i > I $, $ f_{n_i}(K) \cap L = \emptyset $.
\end{defn}

\begin{defn}
Let $f$ be a holomorphic endomorphism of a complex manifold $X$. A point $z \in X$ is said to lie in the {\emph{Fatou set}} $\mathcal{F}$ if there exists a neighborhood $U(z)$ on which the family of iterates $\{f^n\}$ is a normal family. The connected components of $\mathcal{F}$ are called \emph{Fatou components}.
\end{defn}

An important ingredient in the study of one-dimensional Fatou components is the Denjoy-Wolff Theorem \cite{Wolff1}, \cite{Wolff2} \cite{Denjoy}:

\begin{thm}[Wolff-Denjoy]\label{WoDe}
Let $f: \Delta \rightarrow \Delta$ be a holomorphic selfmap of the unit disc. Then either $f$ is a rotation, with respect to Poincar{\'e} metric, about some fixed point $z \in \Delta$, or $f^n$ converges uniformly on compact subsets to a point $z \in \bar{\Delta}$.
\end{thm}

A similar result can be proved for hyperbolic Riemann surfaces, and since every Fatou component of a rational function of degree at least $2$ is hyperbolic, Theorem \ref{WoDe} was an important step towards the classification of invariant Fatou components in one complex dimension.

In an attempt to generalize these results to higher dimensions, Abate studied the dynamics on {\emph{taut}} complex manifolds.

\begin{defn}
A manifold  $\W$ is called taut if $ \mathcal{O}(\Delta,\W) $ is a normal family.
\end{defn}

Let $\Gamma(f)$ be the set of all limits $h \in \mathcal{O}(X, X)$ of convergent subsequences $f^{n_j}$. In \cite{Abate}, Abate proved the following:

\begin{thm}[Abate]\label{Abate}
Let $X$ be a taut manifold and suppose that $f \in \mathcal{O}(X,X)$ is not compactly divergent. Then $\Gamma(f)$ is isomorphic to a compact Abelian group of the form $\Z_q \times \mathbb{T}^r$, where $\Z_q$ is the cyclic group of order $q$, and $\mathbb{T}^r$ is the real torus group of rank $r$.
\end{thm}

Theorem \ref{Abate} gives hope for a classification of invariant Fatou components in higher dimensions, if one could show that invariant Fatou components are taut. This is still an open question. Note that in $ \C^n $ Fatou components are not generally taut; Fatou-Bieberbach domains provide an easy counterexample. The good news is that by now there does exist a classification of invariant Fatou components for holomorphic endomorphisms of $\P^2$. It has become common in the literature to refer to Fatou components on which the family of iterates is not compactly divergent as {\emph{nonrecurrent}}.

\begin{defn}
An invariant Fatou component $\Omega$ is called recurrent if there exists an orbit $f^n(z)$ with an accumulation point in $\Omega$.
\end{defn}

For holomorphic endomorphisms of $\P^2$ recurrent Fatou components have been classified by Forn{\ae}ss and Sibony \cite{FS}, and nonrecurrent Fatou components have been classified by Weickert \cite{Weickert}. We recall the result from \cite{FS}:

\begin{thm}[Fornaess-Sibony] \label{classification}
Suppose that $f$ is a holomorphic self-map of $ \P^2 $ of degree $ d \ge 2 $. Suppose that $ \W $ is a fixed recurrent Fatou component. Then either:

\begin{enumerate}
\item \label{case1} $ \W $ is an attracting basin of some fixed point in $ \W $,
\item \label{case2} there exists a one dimensional closed complex submanifold $ \Sigma $ of $ \W $ and $ f^n(K) \mapsto \Sigma $ for any compact set $ K $ in $ \W $. The Riemann surface $ \Sigma $ is biholomorphic to a disc, a punctured disc, or an annulus and $ f|_\Sigma $ is conjugate to an irrational rotation, or
\item \label{case3} the domain $ \W $ is a Siegel domain.
\end{enumerate}

\end{thm}

Although this classification has been achieved without knowing whether the components are taut, the tautness question remains interesting and may be useful to study finer properties of the dynamical behavior. The only result in this direction is by Weickert \cite{Weickert}, who showed that a Fatou component that is pre-periodic to a basin of attraction (case \ref{case1} in Theorem \ref{classification} above) is taut. We will prove (Theorem \ref{Sigmataut}) that a Fatou component that is pre-periodic to a Fatou component with an attracting Riemann surface (case \ref{case2} in Theorem \ref{classification}) is also taut. Then in Theorem \ref{Siegeltaut} we will prove that a subclass of Siegel domains (case \ref{case3} in Theorem \ref{classification}) is taut as well.

To state our results more precisely we need to consider Siegel discs in greater detail.

\begin{defn}
An invariant Fatou component $\W$ is called a \emph{Siegel domain} there exists a subsequence of iterates $f^{n_j}$ that converges to the identity on $\W$.
\end{defn}

Forn{\ae}ss and Sibony proved the following result:

\begin{prop}[Forn\ae ss -Sibony]\label{Siegel}
Let $W$ be a Siegel domain for $f$, a holomorphic endomorphism of $\P^2$ of degree at least $2$. Denote by $\Gamma_f$ the closure of $(f^n)_{n \in \N}$ in the topology of uniform convergence on compact sets. Then $\Gamma_f$ is a sub-Lie group of $\mathrm{Aut}(\W)$ and $\Gamma_f$ is isomorphic to $\mathbb{T}^k \times F$, where $F$ is a finite group and $k$ is $1$ or $2$.
\end{prop}

Using the notation of the above proposition, we will prove in Theorem \ref{Siegeltaut} that $\W$ is taut when $k = 2$.

For two classes of Fatou components we can prove a stronger property than tautness, namely Kobayashi completeness. Ueda proved in \cite{Ueda} that any Fatou component is Kobayashi Hyperbolic (and even Caratheodory Hyperbolic). A hyperbolic complex manifold is called \emph{Kobayashi complete} if the hyperbolic metric is complete (that is, every Cauchy sequence converges). We will prove in Theorems \ref{basincomplete} and \ref{Sigmacomplete} that a Fatou component that is pre-periodic to case \ref{case1} or case \ref{case2} of Theorem \ref{classification} is Kobayashi complete.

Our organization is as follows. In section (2) we recall some background definitions and prove the tautness of Fatou components with an attracting Riemann surface (Theorem \ref{Sigmataut}). In Section (3) we prove that Fatou components that are either basins of attraction or have an attracting Riemann surface are Kobayashi complete (Theorems \ref{basincomplete} and \ref{Sigmacomplete}). In Section (4) we prove tautness for Siegel domains with a $\mathbb{T}^2$-action (Theorem \ref{Siegeltaut}). In Section (6) we study tautness for recurrent Fatou components of H{\'e}non mappings.

\end{section}

\begin{section}{Tautness for an attracting Riemann surface}

Let $\Omega$ be an invariant Fatou component that has an attracting Riemann surface (case \ref{case2} in Theorem \ref{classification}). In this case, we will prove that $\W$ is taut. We will first recall part of the argument that Weickert used to prove that if $ \W $ is an attracting basin, then $ \W $ is taut, \cite{Weickert}.

Let $f: \P^n \to \P^n $ be a holomorphic map of degree $d$. Then there exists a lift $ F : \C^{n+1} \sm \set{0} \to \C^{n+1} \sm \set{0} $ so that the following diagram commutes, where $ \pi $ is the projection from $ \C^{n+1} \sm \set{0} $ to $ \P^n $, see \cite{FSDiagram}.

\begin{equation}\label{CommutativeDiagram}
\xymatrix{
\C^{n+1} \sm \set{0} \ar[d]^{\pi} \ar[r]^{F}
&\C^{n+1} \sm \set{0} \ar[d]^{\pi} \\
\P^n \ar[r]_{f}
& \P^n
}
\end{equation}

The map $F$ is of the form $(F_1, F_2,..., F_{n+1}) $, where the $F_i$'s are homogeneous polynomials of degree $d$ whose only common zero is $(0,0,0)$. We say $F$ is homogeneous and non-degenerate. If $ d \ge 2 $ it is clear that $F$ has a basin of attraction $ \mathcal{A}$ at the origin. By definition

\begin{equation*}
\mathcal{A} = \set{z \in \C^{n+1} \setminus \{0\} \mid \lim_{n \to \infty} F^n(z) = 0}.
\end{equation*}

It is clear that $ \mathcal{A} $ contains a neighborhood of the origin. Since $F$ is non-degenerate, the basin $\mathcal{A}$ is bounded and by \cite{Ueda} we know the following.

\begin{thm}[Ueda] \label{BoundaryA}
Let $ f : \P^n \to \P^n $. For every point $p$ in $ \mathcal{F}(f) $, the Fatou set of $f$, there exists a neighborhood $V$ of $p$ and a holomorphic map $ s : V \to \C^{n+1} \sm \set{0} $ such that $ \pi \circ s = id $ and $ s(V) \subset \partial \mathcal{A} $. Such a map is unique up to a constant factor of unit length.
\end{thm}

Theorem \ref{BoundaryA} shows that $ \phi \in \mathcal{O} (\Delta, \W) $ lifts to $ s \circ \phi \in \mathcal{O}(\Delta, \partial \mathcal{A}) $.

\begin{thm}\label{Sigmataut}
Let $f$ be a holomorphic self-map of $ \P^2 $ of degree $ d \ge 2 $, and assume that $ \W $ is a preperiodic recurrent Fatou component with an attracting submanifold $ \Sigma $. Then $ \W $ is taut.
\end{thm}

\begin{proof}
It is enough to consider the case $ f: \W \to \W $. Recall from Theorem \ref{classification} of Forn\ae ss and Sibony that we can find a subsequence $ (f^{n_k}) $ of $ (f^n) $ that converges uniformly on compact sets of $\W$ to a map $h: \Omega \rightarrow \Sigma$.

Let a sequence $ (g_i) \subset \mathcal{O}(\Delta,\W) $ be given. Let us assume that $ (g_i) $ is not compactly divergent, so our goal is to prove that there exists a convergent subsequence. Let $F$ denote a lift of $f$ to $ \C^3 \sm \set{0} $ and let $ \mathcal{A} $ be the attracting basin at the origin of $F$. By Theorem \ref{BoundaryA} each $ g_i $ lifts to a map $ \tilde{g_i} : \Delta \to \partial \mathcal{A} $.

Since $\partial \mathcal{A}$ is bounded in $\C^n$ and closed, we can restrict to a subsequence, which for simplicity of notation we will also call $ (\tilde{g_i}) $, that converges uniformly on compacts to a map $\tilde{g} \in \mathcal{O}(\Delta,\partial \mathcal{A})$. Since $ \partial \mathcal{A} $ does not include the origin, we can define $ g = \pi \circ \tilde{g} $ and we obtain

\begin{equation}
\lim_{i \to \infty} g_i = \lim_{i \to \infty} \pi \circ \tilde{g_i} = \pi \circ \tilde{g} = g,
\end{equation}

and this convergence is uniform on compacts, so $ g \in \mathcal{O} (\Delta, \overline{\W}) $.

Now consider the images of $g_i$ and $g$ under $f^n$. Using the lifts to $\C^3\setminus \{0\}$ as above we can restrict to a subsequence of $n_j$ if needed such that the maps $f^{n_j} \circ g \in \mathcal{O}(\Delta, \bar{\W})$ converge to a map $k \in \mathcal{O}(\Delta, \bar{\W})$. Similarly we can assume that for each $i$ the maps $f^{n_j} \circ g_i \in \mathcal{O}(\Delta, \W)$ converge to $h_i \in \mathcal{O}(\Delta, \Sigma)$. By further restricting to a subsequence if necessary we may assume that $h_i$ converges to $h \in \mathcal{O}(\Delta, \bar{\Sigma})$.

To summarize we have

\begin{equation}
k = \lim_{j \to \infty} \pi \circ F^{n_j} \circ \tilde{g} = \lim_{j \to \infty} f^{n_j} \circ \pi \circ \tilde{g} = \lim_{j \to \infty} f^{n_j} \circ g,
\end{equation}

and

\begin{equation}
h_i = \lim_{j \to \infty} \pi \circ F^{n_j} \circ \tilde{g_i} = \lim_{j \to \infty} f^{n_j} \circ \pi \circ \tilde{g_i} = \lim_{j \to \infty} f^{n_j} \circ g_i.
\end{equation}

By our assumption that $g_i$ is not compactly divergent there exists $ \zeta \in \Delta $ so that $ g(\zeta) \in \W $. Take a neighborhood $ U \subset \W $ containing $ g(\zeta) = z $ so that $f^{n_j}$ converges uniformly on $U$. Then there exists a neighborhood $V \subset \Delta$ of $\zeta$ such that $g_i(V) \subset U$ for all $i$ large enough. By the definitions of $h$ and $k$ it follows immediately that $h = k$ on $V$, and since the two maps are holomorphic we have that $h = k$ on the whole disc $\Delta$.

Now let us consider the sequence of discs $h_i \in \mathcal{O}(\Delta, \Sigma)$. As was shown in \cite{FS} and \cite{Ueda2} the Riemann surface $\Sigma$ must be biholomorphically equivalent to either the unit disc or to the annulus, and both are taut. Therefore the family $\{h_i\} \subset \mathcal{O}(\Delta, \Sigma)$ is either compactly divergent or admits a convergent subsequence. But since $\Sigma$ is a closed submanifold of $\W$, if the family is compactly divergent in $\Sigma$ it must also be compactly divergent in $\W$, which would imply that the maps $g_i$ were already compactly divergent, which contradicts our assumptions. Therefore we have that $\{h_i\}$ admits a normal family and the limit $h$ maps $\Delta$ into $\Sigma$. Since $h =k$ it follows that $g$ maps $ \Delta $ into $ \W $ which concludes the argument.
\end{proof}

\end{section}

\begin{section}{Kobayashi Completeness}

Barth proved that on hyperbolic complex spaces the topology induced by the Kobayashi distance is the same as the standard topology, see \cite{Barth}.

\begin{thm}[Barth]
Let $X$ be a connected complex space. Then the Kobayashi distance is continuous. If $ B_K (p,r) $ denotes the open Kobayashi ball of radius $r$ around $ p \in X $, then $ B_K (p,r) $ is path-connected in $X$. If $X$ is Kobayashi hyperbolic, then the Kobayashi distance induces the standard topology on $X$. In fact, for any open neighborhood $U$ in the standard topology and any point $ p \in U $ there is a Kobayashi ball around $p$ that is compactly contained in $U$.
\end{thm}

\begin{defn}
A hyperbolic complex manifold $ D \subset \C^n $ is complete if every Kobayashi Cauchy sequence of points $ (z_i) $ with $ z_i \in D $ converges in the usual topology to a point $ z \in D $.
\end{defn}

\begin{thm}\label{basincomplete}
Let $f$ be a holomorphic self-map of $\P^n$ of degree at least $2$. Let $\W$ be a Fatou component of $f$ which is preperiodic to a basin of attraction, with attracting fixed point $ p \in \W $. Then $\W$ is Kobayashi complete.
\end{thm}

\begin{proof}
It is enough to consider an invariant Fatou component. Let $ (z_i) $ be a Kobayashi Cauchy sequence in $\W$. Let $U$ be a small neighborhood of the attracting fixed point $p$ and let $\e>0$ be small enough such that $\mathcal{N}_\e(U)$, the $\e$-neighborhood of $U$ in the Kobayashi-metric, is relatively compact in $\W$. Let $I \in \N$ be such that $d_K(z_i, z_j) < \e$ for any $i, j \ge I$. Since $f^n \rightarrow p$ locally uniformly on $\W$ there exists an $N \in \N$ such that $f^N(z_I) \in U$. Since $d_k(f^N(z_j), f^N(z_I) ) \le d_k(z_j, z_I)$ we have that $z_j \in \mathcal{N}_\e(U)$ for any $j \ge I$. But by continuity of $f$ we have that $K = f^{-N}(\mathcal{N}_\e(U)) \cap \W$ is relatively compact in $\W$. By compactness of $\P^2$ there exists a convergent subsequence $z_{i_j} \rightarrow z$, where $z$ must necessarily lie in $\bar{K} \subset \W$. But since $d_K(z_i, z) \le d_K(z_i, z_{i_j}) + d_K(z_{i_j}, z)$ we have that $d_K(z_i, z) \rightarrow 0$. Since the topology defined by the Kobayashi metric is the same as the usual topology we get that $ (z_i) $ converges to $z$.
\end{proof}

\begin{defn}
A holomorphic retract $ f : U \to U $ is a holomorphic map that satisfies $ f(f(z)) = f(z) $ for all $ z \in U $.
\end{defn}

Let $f$ be a holomorphic endomorphism of $\P^2$ of degree at least $2$. Let $\W$ be an invariant Fatou component of $f$ and assume that all orbits in $\W$ converge to a Riemann surface $\Sigma \in \W$. Since $f$ acts as a rotation on $\Sigma$, there exists a subsequence $\{f^{n_j}\}$ that converges uniformly on compact subsets of $\W$ to $h: \W \rightarrow \Sigma$ and such that $h|_\Sigma = \mathrm{Id}$. This $h$ is a holomorphic retract.

In order to prove Kobayashi completeness in case \ref{case2} of Theorem \ref{classification} we need the following lemma.

\begin{lemma}\label{equalmetric}
Let $ f: U \to V \subset U $ be a holomorphic retract. Then for any points $ x, y \in V $, $ d_K^U (x,y) = d_K^V (x,y) $.

\end{lemma}

\begin{proof}

Since $f$ is holomorphic, using the property that the Kobayashi distance is decreasing under holomorphic maps it is clear that $d_K^U(z,w) \ge d_K^V(z,w)$. On the other hand, by inclusion $d_K^U(z,w) \le d_K^V(z,w)$.

\end{proof}

We can now prove the following:

\begin{thm}\label{Sigmacomplete}
Let $f$ be a holomorphic endomorphism of $\P^2$ of degree at least $2$. Let $\W$ be an invariant Fatou component of $f$ and assume that all orbits in $\W$ converge to a Riemann surface $\Sigma \in \W$. Then $\W$ is Kobayashi complete.
\end{thm}
\begin{proof}
Let the subsequence $ (f^{n_j}) $ converge to a map $h: \W \rightarrow \Sigma$, uniformly on compact subsets of $\W$. Suppose that $ (z_i) $ is a Kobayashi Cauchy sequence in $ \W $. Let $w_i = h(z_i)$. If $\{w_i\}$ lies in a relatively compact subset of $\Sigma$ then we can apply the same argument as in the proof of Theorem  \ref{basincomplete} to obtain that $ (z_i) $ converges to a point $z \in \W$. Let us therefore assume, for the purpose of contradiction, that $\{w_i\}$ is not relatively compact in $\Sigma$.

We know from the work of Forn{\ae}ss and Sibony \cite{FS} that $\Sigma$ is a closed invariant submanifold and by Ueda \cite{Ueda2} that $\Sigma$ is biholomorphically equivalent to the disc or an annulus. Since the disc and the annulus are both Kobayashi complete, we get that $ (w_i) $ cannot be a Cauchy sequence in $\Sigma$ with respect to the Kobayashi distance on $\Sigma$. But by Lemma \ref{equalmetric} we conclude that $ (w_i) $ cannot be a Cauchy sequence with respect to the Kobayashi distance on $\W$ either. Since the Kobayashi distance is continuous and $d_K(f^n(z_i), f^n(z_j)) \le d_K(z_i, z_j)$ we have that $d_K(h(z_i), h(z_j)) \le d_K(z_i, z_j)$ as well, and therefore $ (w_i) $ must be a Cauchy sequence by our assumption that $ (z_i) $ is a Cauchy sequence. This contradiction concludes our argument.
\end{proof}

\end{section}

\begin{section}{Siegel domains}

As we mentioned in the introduction (Proposition \ref{Siegel}) it was shown by Forn\ae ss and Sibony that a Siegel disc either admits a $\mathbb{T}^1$ or a $\mathbb{T}^2$ action. In the latter case it follows from the following short argument that the Siegel disc is taut.

\begin{thm}\label{Siegeltaut}
Let $f$ be a holomorphic endomorphism of $\P^2$ of degree at least $2$, such that $f$ has a Siegel domain $\W$. Further assume that $\Gamma_f$ is isomorphic to $\mathbb{T}^2 \times F$. Then $\Omega$ is taut.
\end{thm}
\begin{proof}
Since $\Omega$ is pseudoconvex (\cite{Ueda}, \cite{FSDiagram}) and admits a $2$-torus action, a result by Barrett, Bedford and Katok \cite{BBK} gives that $\Omega$ is biholomorphic to a Reinhardt domain $V$. But then $V$ is a hyperbolic, pseudoconvex Reinhardt domain and therefore taut, see \cite{Zwonek}.
\end{proof}

The Siegel disc with a $\mathbb{T}^1$-action is the only remaining recurrent Fatou component that is not covered yet. It would be interesting to know whether this type of Fatou component must be taut as well.

\end{section}

\begin{section}{Tautness for Fatou components of H{\'e}non mappings}

In \cite{FS} recurrent Fatou components were studied for (generalized) H{\'e}non mappings as well as for Holomorphic mappings of $\P^2$. Forn{\ae}ss and Sibony showed that the classification in Theorem \ref{classification} also holds for H{\'e}non maps. These Fatou components are often not Kobayashi hyperbolic though, for example a basin of attraction $\W$ is always biholomorphic to $\C^2$. For the same reason such a Fatou component cannot be taut, for any $z \in \W$ there will be families of holomorphic discs $\phi_n: \Delta \rightarrow \W$ for which $\phi_n(0) = z$ but such that the set $\{\phi(\frac{1}{2})\}$ is not contained in any bounded subset of $\C^2$. Such a family is not compactly divergent but also clearly does not converge. This example suggests that for unbounded Fatou components in $\C^2$ we should only consider bounded families of holomorphic discs:

\begin{defn}
A domain $\W \subset \C^n$ is called \emph{taut on bounded families} if every bounded family $V \subset \mathcal{O}(\Delta, \W)$ is normal.
\end{defn}

\begin{thm}
Let $\W$ be a recurrent Fatou component for a (generalized) H{\'e}non map $f$. Suppose that $\W$ is either an attracting basin (corresponding to case \ref{case1} in Theorem \ref{classification}) or has an attractive Riemann surface (corresponding to case \ref{case2} in Theorem \ref{classification}). Then $\W$ is taut on bounded families.
\end{thm}

The proofs are essentially the same as the proofs of Theorem 2 in \cite{Weickert} and of Theorem \ref{Sigmataut}. Similarly the proof of Theorem \ref{Siegeltaut} immediately gives:

\begin{thm}
Let $\W$ be a recurrent Fatou component for a (generalized) H{\'e}non map $f$. Suppose that $\W$ is a Siegel disc that admits a $\mathbb{T}^2$ action. Then $\W$ is taut. 
\end{thm}

\end{section}

\bigskip

\noindent Han Peters,\\
Korteweg de Vries Instituut voor Wiskunde \\
Universiteit van Amsterdam \\
P.O. Box 94248 \\
1090 GE AMSTERDAM \\
Netherlands\\
h.peters@uva.nl \\

\noindent Crystal Zeager\\
Mathematics Department\\
University of Michigan\\
East Hall, Ann Arbor, MI 48109\\
USA\\
zeagerc@umich.edu\\

\newpage


\begin{thebibliography}{9999}

\bibitem{Abate} Abate, Marco, \emph{Iteration theory, compactly divergent sequences and commuting
              holomorphic maps}, Ann. Scuola Norm. Sup. Pisa Cl. Sci. (4), {\bf 18} (1991), no. 2, 167--191.

\bibitem{BBK} Barrett, D. E., Bedford, E. and Dadok, J., \emph{{$T^n$}-actions on holomorphically separable complex
              manifolds}, Math. Z., {\bf 202} (1989), no. 1, 65--82.

\bibitem{Barth} Barth, T. J., The Kobayashi distance induces the standard topology, Proc. Amer. Math. Soc., {\bf 35} (1972), 439-441.

\bibitem{Denjoy} Denjoy, A., \emph{Sur l'it{\'e}ration des fonctions analitiques}, C. R. Acad. Sci. Paris, {\bf 182} (1926), 255--257.

\bibitem{FSDiagram} Fornaess, J. E., Sibony, N. \emph{Complex dynamics in higher dimensions: 1, Complex analytic methods in dynamical systems (Rio de Janeiro, 1992)}, Ast{\'e}risque, {\bf 222} (1994), no. 5, 201--231.

\bibitem{FS} Fornaess, J. E., Sibony, N., \emph{Classification of recurrent domains for some holomorphic maps}, Math. Ann., {\bf 301} (1995), no. 4, 813-820.

\bibitem{HP} Hubbard, J., Papadopol, P., \emph{Superattractive fixed points in $\C^n$}, Indiana Univ. Math. J., {\bf 43} (1994), 321--365.

\bibitem{JP} Jarnicki, M., Pflug, P. \emph{Invariant Distances and Metrics in Complex Analysis}, De Gruyter Expositions in Mathematics, {\bf 9}, (1993).



\bibitem{LZ} Lee, L., Zeager, C. \emph{some title}, to appear.

\bibitem{Ueda} Ueda, T. \emph{Fatou sets in complex dynamics on projective spaces}, J. Math. Soc. Japan, {\bf 46} (1994), 545-555.

\bibitem{Ueda2} Ueda, T., \emph{Holomorphic maps on projective spaces and continuations of Fatou maps}, Michigan Math J., {\bf 56}, (2008), no. 1, 145-153.

\bibitem{Weickert} Weickert, B. J., \emph{Nonwandering, nonrecurrent Fatou components in $ \P^2 $}, Pac. J. Math., {\bf 211}, (2003), no. 2, 391--397.

\bibitem{Wolff1} Wolff, J., \emph{Sur l'it{\'e}ration des fonctions holomorphes dans une r{\'e}gion, et dont les valuers appartiennent \`a cette r{\'e}gion}, C. R. Acad. Sci. Paris, {\bf 182} (1926), 42--43.

\bibitem{Wolff2} Wolff, J., \emph{Sur l'it{\'e}ration des fonctions born{\'e}es}, C. R. Acad. Sci. Paris, {\bf 182} (1926), 200--201.

\bibitem{Zwonek} Zwonek, W., \emph{On hyperbolicity of pseudoconvex Reinhardt domains}, Arch. Math., {\bf 72}, (1999), 304-314.

\end{thebibliography}
\end{document}